\documentclass{amsart} 
\usepackage{amscd}
\usepackage{amsmath,empheq}
\usepackage{amsfonts}
\usepackage{amssymb}
\usepackage{mathrsfs}
\usepackage[all]{xy}
\usepackage{mathtools}
\newtheorem{theorem}{Theorem}[section]
\newtheorem{ex}{Example}[section]
\newtheorem{lemma}[theorem]{Lemma}
\newtheorem{prop}[theorem]{Proposition}
\newtheorem{remark}{Remark}[section]
\newtheorem{corollary}[theorem]{Corollary} 
\newtheorem{claim}{Claim}[section]

\newenvironment{proof-sketch}{\noindent{\bf Sketch of Proof}\hspace*{1em}}{\qed\bigskip}

\everymath{\displaystyle}
\newcommand{\RR}{\mathbb R}
\newcommand{\NN}{\mathbb N}

\renewcommand{\leq}{\leqslant}

\renewcommand{\geq}{\geqslant}
\begin{document}
\title[Nonlinear singular problems with indefinite potential term]{Nonlinear singular problems with \\ indefinite potential term}
\author[N.S. Papageorgiou]{N.S. Papageorgiou}
\address[N.S. Papageorgiou]{ Department of Mathematics, National Technical University,
				Zografou Campus, 15780 Athens, Greece \& Institute of Mathematics, Physics and Mechanics, 1000 Ljubljana, Slovenia}
\email{\tt npapg@math.ntua.gr}
\author[V.D. R\u{a}dulescu]{V.D. R\u{a}dulescu}
\address[V.D. R\u{a}dulescu]{Faculty of Applied Mathematics, AGH University of Science and Technology, 30-059 Krak\'ow, Poland
 \& Institute of Mathematics, Physics and Mechanics, 1000 Ljubljana, Slovenia \& Department of Mathematics, University of Craiova, 200585 Craiova,, Romania}
\email{\tt vicentiu.radulescu@imfm.si}
\author[D.D. Repov\v{s}]{D.D. Repov\v{s}}
\address[D.D. Repov\v{s}]{Faculty of Education and Faculty of Mathematics and Physics, University of Ljubljana \& Institute of Mathematics, Physics and Mechanics, 1000 Ljubljana, Slovenia}
\email{\tt dusan.repovs@guest.arnes.si}
\begin{abstract}
We consider a nonlinear Dirichlet problem driven by a nonhomogeneous differential operator plus an indefinite potential. In the reaction we have the competing effects of a singular term and of concave and convex nonlinearities. In this paper the concave term is parametric. We prove a bifurcation-type theorem describing the changes in the set of positive solutions as the positive parameter $\lambda$ varies. This work continues our research published in arXiv:2004.12583, where $\xi \equiv 0 $ and in the reaction the parametric term is the singular one.
\end{abstract}
\maketitle

\section{Introduction}\label{sec1}

Let $\Omega \subseteq \RR^{N} $ be a bounded domain with a $C^2$-boundary $\partial\Omega$. In this paper we study the following nonlinear nonhomogeneous parametric singular problem:
\begin{equation}
    \left\{
        \begin{array}{ll}
          -  \mbox{div}\, a(Du(z)) + \xi(z)u(z)^{p-1} = \vartheta(u(z)) + \lambda u(z)^{q-1} + f(z,u(z)) \ \mbox{in}\ \Omega, \\
            u|_{\partial\Omega} = 0,\ u>0,\ \lambda > 0,\ 1<q<p<\infty.
        \end{array}
    \right\}\tag{$P_{\lambda}$}\label{eqp}
\end{equation}

The map $a:\RR^{N} \rightarrow \RR^{N}$ involved in the differential operator of \eqref{eqp} is strictly monotone, continuous (hence maximal monotone, too) and satisfies certain other regularity and growth conditions which are listed in hypotheses $H(a)$ below (see Section \ref{sec2}). These conditions are general enough to incorporate in our framework many differential operators of interest such as the $p$-Laplacian and the $(p,q)$-Laplacian (that is, the sum of a $p$-Laplacian and a $q$-Laplacian). The operator $u\mapsto{\rm div}\,a(Du)$ is not homogeneous and this is a source of difficulties in the analysis of problem \eqref{eqp}. The potential function $\xi \in L^{\infty}(\Omega)$ is indefinite (that is, sign changing). So the operator $u \mapsto -{\rm div}\, a(Du) + \xi (z) |u|^{p-2} u $ is not coercive and this is one more difficulty in the analysis of problem \eqref{eqp}. In the reaction (the right-hand side of \eqref{eqp}), the term $\vartheta(\cdot)$ is singular at $x=0$, while the perturbation contains the combined effects of a parametric concave term $x\mapsto \lambda x^{q-1}$ ($x\geq 0 $) (recall that $q<p$), with $\lambda >0$ being the parameter and of a Carath\'eodory function $f(z,x)$ (that is, for all $x\in \RR$ the mapping
  $z\mapsto f(z,x)$ is measurable and for almost all $z\in \Omega$ the mapping $ x\mapsto f(z,x)$ is continuous), which is assumed to exhibit $(p-1)$-superlinear growth near $+\infty$, but without satisfying the usual for superlinear problems Ambrosetti-Rabinowitz condition (the AR-condition for short). So in problem \eqref{eqp} we have the competing effects of singular, concave and convex terms.

  Using variational methods related to the critical point theory, combined with suitable truncation, perturbation and comparison techniques, we produce a critical parameter value $\lambda^* >0$ such that
\begin{itemize}
    \item [(i)] for all $\lambda\in (0,\lambda^*)$ problem \eqref{eqp} has at least two positive solutions;
    \item [(ii)] for $\lambda = \lambda^*$ problem \eqref{eqp} has at least one positive solution;
    \item [(iii)] for all $\lambda > \lambda^*$ problem \eqref{eqp} has no positive solutions.
\end{itemize}

This work continues our research published in Papageorgiou, R\u{a}dulescu \& Repov\v{s} \cite{16}, where $\xi \equiv 0 $ and in the reaction the parametric term is the singular one (cf. arXiv:2004.12583). It is also related to the work of Papageorgiou \& Smyrlis \cite{17} and Papageorgiou \& Winkert \cite{18}, where the differential operator is the $p$-Laplacian, $\xi\equiv0$ and no concave terms are allowed. Singular $p$-Laplacian equations with no potential term and reactions of special form were considered  by Chu, Gao \& Sun \cite{2}, Giacomoni, Schindler \& Taka\v c \cite{5}, Li \& Gao \cite{11}, Mohammed \cite{12}, Perera \& Zhang \cite{20}, and Papageorgiou, R\u{a}dulescu \& Repov\v{s} \cite{14}.

\section{Mathematical background and hypotheses}\label{sec2}

In this section we present the main mathematical tools which we will use in the analysis of problem \eqref{eqp}. We also fix our notation and state the hypotheses on the data of the problem.

So, let $X$ be a Banach space, $X^*$ its topological dual, and let $\varphi \in C^1(X).$ We say that $\varphi(\cdot)$ satisfies the ``C-condition", if the following property holds:
$$
\begin{array}{ll}
``\mbox{Every sequence}\ \{u_n\}_{n\geq1}\subseteq X\ \mbox{such that}\\
\{\varphi(u_n)\}_{n\geq1}\subseteq\RR\ \mbox{is bounded and}
\ (1+||u_n||_X)\varphi'(u_n)\rightarrow0\ \mbox{in}\ X^*\ \mbox{as}\ n\rightarrow\infty,\\
\mbox{admits a strongly convergent subsequence}".
\end{array}
$$

This is a compactness-type condition on the functional $\varphi(\cdot)$, which leads to the minimax theory of the critical values of $\varphi(\cdot)$ (see, for example, Papageorgiou, R\u{a}dulescu \& Repov\v{s} \cite{15}). We denote by $K_\varphi$ the critical set of $\varphi$, that is,
$$ K_{\varphi} = \{ u\in X: \varphi'(u)=0\}. $$

The main spaces in the analysis of problem \eqref{eqp} are the Sobolev space $W^{1,p}_0(\Omega)$ $(1<p<\infty)$ and the Banach space $C_0^1(\overline{\Omega})=\{u\in C^1(\overline{\Omega}): u|_{\partial\Omega} = 0\}$. We denote by $||\cdot||$ the norm of $W^{1,p}_0$. By the Poincar\'e inequality we have
$$ ||u|| = ||Du||_p\ \mbox{for all}\ u \in W^{1,p}_0(\Omega). $$

The Banach space $C_0^1(\Omega)$ is ordered with positive (order) cone $$C_+ = \{u\in C_0^1(\overline{\Omega}): u(z)\geq 0\ \mbox{for all}\ z\in\overline{\Omega} \}.$$ This cone has a nonempty interior given by
$$
\begin{array}{ll}
        {\rm int}\, C_+ = \left\{u\in C_+ : u(z) >0\ \mbox{for all}\ z\in \Omega ,\ \frac{\partial u}{\partial n} |_{\partial\Omega} <0 \right\} , \\
        \mbox{with}\ n(\cdot)\ \mbox{being the outward unit normal on}\ \partial\Omega.
\end{array}
$$

We will also use two additional ordered Banach spaces. The first one is
$$ C_0(\overline{\Omega}) = \{ u\in C(\overline{\Omega}):u|_{\partial\Omega} = 0 \}.$$
This cone is ordered with positive (order) cone $$K_+ = \{ u \in C_0 (\overline{\Omega}): u(z) \geq 0\ \mbox{for all $z\in\overline{\Omega}$}\}.$$ This cone has a nonempty interior given by
$$ {\rm int}\,K_+ = \{ u \in K_+: c_u \hat{d} \leq u \ \mbox{for some}\ c_u>0 \}, $$
where $\hat{d}(z) = d(z,\partial\Omega)\ \mbox{for all}\ z\in\overline{\Omega}$. On account of Lemma 14.16 of Gilbarg \& Trudinger \cite[p. 355]{6},  we have
\begin{equation}
      ``c_u \hat{d} \leq u \ \mbox{for some}\ c_u > 0\ \mbox{if and only if}\ \hat{c}_u \hat{u}_1 \leq u\ \mbox{for some} \ \hat{c}_u > 0 ",
    \label{eq1}
\end{equation}
with $\hat{u}_1$ being the positive, $L^p$-normalized (that is, $||\hat{u}_1||_p=1$) eigenfunction corresponding to the principal eigenvalue $\hat{\lambda}_1>0$ of the Dirichlet $p$-Laplacian. The nonlinear regularity theory and the nonlinear maximum principle (see, for example, Gasinski \& Papageorgiou \cite[pp.737-738]{4}), imply that $\hat{u}_1\in{\rm int}\,C_+$.

The second ordered space is $C^1(\overline{\Omega})$ with positive (order) cone
$$ \hat{C}_+ = \left\{ u \in C^1(\overline{\Omega}): u(z) \geq 0 \ \mbox{for all}\ z\in\overline{\Omega}, \ \frac{\partial u}{\partial n}|_{\partial\Omega \cap u^{-1}(0)} <0 \right\}.$$
Clearly, this cone has a nonempty interior.

Concerning ordered Banach spaces with an order cone which has a nonempty interior (solid order cone), we have the following result which will be useful in our analysis (see Papageorgiou, R\u{a}dulescu \& Repov\v{s} \cite[Proposition 4.1.22]{15}).

\begin{prop}\label{prop1}
    If $X$ is an ordered Banach space with positive (order) cone K, int\,K $\neq \emptyset$, and $e \in int\,K$, then for every $u\in X$ we can find $\lambda_u >0$ such that $\lambda_u e - u\in K.$
\end{prop}

Let $l\in C^1(0,\infty)$ with $l(t)>0$ for all $t>0$. We assume that
\begin{equation}
    \begin{array}{ll}
        0<\hat{c}\leq\frac{l'(t)t}{l(t)}\leq c_0, c_1 t^{p-1}\leq l(t) \leq c_2[t^{s-1}+ t^{p-1}] \\
        \mbox{for all}\ t>0,\ \mbox{and some}\ c_{1},c_{2}>0, 1\leq s< p.
    \end{array}
    \label{eq2}
\end{equation}

Then the conditions on the map $a(\cdot)$ are the following:

\smallskip
$H(a): a(y) = a_0(|y|)y$ for all $y\in\RR^{N}$, with $a_0(t)>0$ for all $t>0$ and
\begin{itemize}
    \item [(i)] $a_0\in C^1(0,+\infty),\ t\mapsto a_0 (t)$ is strictly increasing on $(0,+\infty)$, $a_0(t) t\rightarrow 0^+$ as $t\rightarrow 0^+$ and
        $$ \lim_{t\to 0^+}\frac{a'_0(t)t}{a_0(t)} > -1; $$
    \item [(ii)] there exists $c_3>0$ such that
        $$ |\nabla a(y)| \leq c_3 \frac{l(|y|)}{|y|}\ \mbox{for all}\ y\in \RR^{N} \backslash \{0\}; $$
    \item [(iii)] $(\nabla a(y)\xi,\xi)_{\RR^{N}} \geq \frac{l(|y|)}{|y|}|\xi|^2$ for all $y\in\RR^{n} \backslash \{0\}$, $\xi\in\RR^{N}$;
    \item [(iv)] if $G_0(t) = \int_{0}^{t} a_0(s)s ds $, then there exists $\tau\in (q,p]$ such that
        $$ \limsup_{t\to 0^+} \frac{\tau G_0(t)}{t^{\tau}} \leq c^* $$
\end{itemize}
and $0\leq p G_0(t) - a_0(t) t^2$ for all $t>0$.
\begin{remark}\label{rem1}
    Hypotheses $H(a)(i),(ii),(iii)$ are dictated by the nonlinear regularity theory of Lieberman [10] and the nonlinear maximum principle of Pucci \& Serrin \cite{21}. Hypothesis $H(a)(iv)$ serves the needs of our problem, but in fact, it is a mild condition and it is satisfied in all cases of interest (see the examples below). These conditions were used by Papageorgiou \& R\u{a}dulescu \cite{13} and by Papageorgiou, R\u{a}dulescu \& Repov\v{s} \cite{16}.
\end{remark}

Hypotheses $H(a)$ imply that the primitive $G_0(\cdot)$ is strictly increasing and strictly convex. We set $G(y) = G_0(|y|)$ for all $y\in \RR^{N}.$ Evidently $G(\cdot)$ is convex, $G(0) = 0$ and
$$
\begin{array}{ll}
    \nabla G(y) = G_0'(|y|)\frac{y}{|y|} = a_0 (|y|)y = a(y)\ \mbox{for all}\ y\in\RR^{N}\backslash\{0\},\ \nabla G(0)=0,\\
    \mbox{that is},\ G(\cdot) \ \mbox{is the primitive of}\ a(\cdot).\ \mbox{From the convexity of}\ G(\cdot)\ \mbox{we have}
\end{array}
$$
\begin{equation}
    G(y)\leq(a(y),y)_{\RR^{N}}\ \mbox{for all}\ y\in\RR^{N}.
    \label{eq3}
\end{equation}

Using hypotheses $H(a)(i),(ii),(iii)$ and \eqref{eq2}, we can easily obtain the following lemma, which summarizes the main properties of the map $a(\cdot)$.
\begin{lemma}\label{lem2}
If hypotheses $H(a)(i),(ii),(iii)$ hold, then
\begin{itemize}
    \item [(a)] the map $y\mapsto a(y)$ is continuous, strictly monotone (hence maximal monotone, too);
    \item [(b)] $|a(y)| \leq c_4 (|y|^{s-1} +|y|^{p-1})$ for some $c_4>0$, and all $y\in\RR^{N}$;
    \item [(c)] $(a(y),y)_{\RR^{N}} \geq \frac{c_1}{p-1} |y|^p$ for all $y\in\RR^{N}$.
\end{itemize}
\end{lemma}

Using this lemma and \eqref{eq3}, we obtain the following growth estimates for the primitive $G(\cdot)$.
\begin{corollary}\label{cor3}
If hypotheses $H(a)(i),(ii),(iii)$ hold, then $\frac{c_1}{p(p-1)} |y|^{p} \leq G(y) \leq c_5 (1+|y|^p) $ for some $c_5>0$, and all $y\in\RR^{N}$.
\end{corollary}

The examples that follow confirm that the framework provided by hypotheses $H(a)$ is broad and includes many differential operators of interest (see \cite{13}).
\begin{ex}\label{ex1}
	\begin{itemize}
       \item [(a)] $a(y) = |y|^{p-2}y$ with $1<p<\infty$.
       \item [] This map corresponds to the $p$-Laplace differential operator defined by
        $$ D_p u = {\rm div}\, (|Du|^{p-2} Du)\ \mbox{for all}\ u\in W^{1,p}_0 (\Omega). $$
       \item [(b)] $a(y)=|y|^{p-2} y + \mu|y|^{q-2} y$ with $1<q<p<\infty,\ \mu \geq 0$.
       \item [] This map corresponds to the $(p,q)$-Laplace differential operator defined by
        $$ D_p u + D_q u \ \mbox{for all}\ u\in W^{1,p}_0(\Omega). $$
       \item [] Such operators arise in models of physical processes. We mention the works of Cherfils \& Ilyasov \cite{1} (reaction-diffusion systems) and Zhikov \cite{22} (homogenization of composites consisting of two materials with distinct hardening exponents, in elasticity theory).
        \item [(c)] $a(y) =(1+|y|^2)^{\frac{p-2}{2}}y$ with $1<p<\infty$.
        \item []This map corresponds to the modified capillary operator.
        \item [(d)] $a(y) = |y|^{p-2}y\left(1+ \frac{1}{1+ |y|^p}\right)$ with $1<p<\infty.$
	\end{itemize}
\end{ex}

The hypotheses on the potential term $\xi(\cdot)$ and on the singular part $\vartheta(\cdot)$ of the reaction are the following:

\smallskip
$H(\xi): \xi\in L^\infty(\Omega).$

\smallskip
$H(\vartheta): \vartheta:(0,+\infty)\rightarrow (0,+\infty)$ is a locally Lipschitz function such that
\begin{itemize}
    \item [(i)] for some $\gamma\in(0,1)$ we have
        $$ 0 < c_6\leq \liminf_{x\to 0^+}\vartheta(x)x^{\gamma} \leq \limsup_{x\to 0^+}\vartheta(x) x^\gamma \leq c_7; $$
    \item [(ii)] $\vartheta(\cdot)$ is nonincreasing.
\end{itemize}

\begin{remark}\label{rem2}
In the literature we almost always encounter the following particular singular term
$$ \vartheta(x) = x^{-\gamma}\ \mbox{for all}\ x>0, \ \mbox{with}\ 0<\gamma<1. $$
\end{remark}

Of course, hypotheses $H(\vartheta)$ provide a much more general framework and can accomodate also singularities like the ones that follow:
\begin{eqnarray*}
    &&\vartheta_1(x)=x^{-\gamma}\left[1+\ln(1+x)\right],\quad x>0,  \\
    &&\vartheta_2(x)=x^{-\gamma}e^{-x},\quad x>0 \\
    &&\vartheta_3(x)= \left\{
            \begin{array}{ll}
                    x^{-\gamma}(1-\eta \sin x )& 0<x\leq \frac{\pi}{2}\\
                    x^{-\gamma}(1-\eta)& \ \mbox{if}\ \frac{\pi}{2} <x
            \end{array}
            \ \mbox{with}\  0<\gamma<1.
        \right.
\end{eqnarray*}

The following strong comparison principle can be found in Papageorgiou, R\u{a}dulescu \& Repov\v{s} \cite[Proposition 6]{16} (see also Papageorgiou \& Smyrlis \cite[Proposition 4]{17}).

\begin{prop}\label{prop4}
    If hypotheses $H(a),H(\vartheta)$ hold, $\hat{\xi}\in L^{\infty}(\Omega)$, $\hat{\xi}(z) \geq 0$  for almost all $z\in \Omega$, $h_1,h_2 \in L^{\infty}(\Omega)$ satisfy
    $$ 0<c_8\leq h_2(z) - h_1(z) \ \mbox{for almost all}\ z\in\Omega$$
    and $u,v\in C^{1,\alpha}(\overline{\Omega})$ satisfy $0<u(z)\leq v(z)$ for all $z\in\Omega$ and for almost all $z\in\Omega$ we have
    \begin{itemize}
        \item $-{\rm div}\, a(Du(z)) - \vartheta(u(z)) + \xi(z) u (z)^{p-1} = h_1(z)$
        \item $- {\rm div}\, a(Dv(z) - \vartheta(v(z)) + \xi(z) v(z)^{p-1} = h_2(z)$,
    \end{itemize}
    then $v-u\in {\rm int}\,\hat{C}_+.$
\end{prop}

In what follows, $p^*$ is the critical Sobolev exponent corresponding to $p$, that is,
$$
p^* =\left\{
        \begin{array}{ll}
            \frac{Np}{N-p}& \mbox{if}\ p<N.\\
            +\infty& \mbox{if}\ N\leq p.
        \end{array}
    \right.
$$
Now we introduce our hypotheses on the nonlinearity $f(z,x)$.

\smallskip
$H(f): f:\Omega\times\RR\rightarrow\RR_+$ is a Carath\'eodory function such that $f(z,0)=0$ for almost all $z\in\Omega$ and
\begin{itemize}
    \item [(i)] $f(z,x) \leq a(z) (1+x^{r-1})$ for almost all $z\in\Omega$, and all $x\geq 0$, with $a\in L^{\infty}(\Omega)$, $p<r<p^*;$
    \item [(ii)] if $F(z,x)=\int_0^x f(z,s) ds,$
    \item [] then $\lim_{x\to+\infty}\frac{F(z,x)}{x^{p}} = +\infty$ uniformly for almost all $z\in\Omega$;
    \item [(iii)] there exists $\sigma\in((r-p) \max\{ \frac{N}{p}, 1 \},p^*),\ \sigma>q$ such that
		$$ 0<\hat{\beta}_0\leq \liminf_{x\to+\infty} \frac{f(z,x)x-pF(z,x)}{x^{\sigma}}\ \mbox{uniformly for almost all}\ z\in\Omega; $$
    \item [(iv)] $\limsup_{x\rightarrow 0^+} \frac{f(z,x)}{x^{r-1}} \leq \eta_0 $ uniformly for almost all $z\in\Omega$;
    \item [(v)] for every $\rho>0$, there exists $\hat{\xi}_\rho>0$ such that for almost all $z\in\Omega$ the function
    $$ x\mapsto f(z,x) + \hat{\xi}_\rho x^{\rho-1} $$
    is nondecreasing on $[0,\rho]$.
\end{itemize}
\begin{remark}\label{rem3}
    Since our aim is to find positive solutions and the above hypotheses concern the positive semiaxis $\RR_+ = [0,+\infty)$, we may assume that
    \begin{equation}
        f(z,x) = 0,\ \mbox{for almost all}\ z\in\Omega, \ \mbox{and all}\ x\leq 0.
        \label{eq4}
    \end{equation}
\end{remark}

Hypotheses $H(f)(ii),(iii)$ imply that
\begin{equation}
    \lim_{x\to+\infty} \frac{f(z,x)}{x^{p-1}} = +\infty \ \mbox{uniformly for almost all}\ z\in\Omega.
    \label{eq5}
\end{equation}

So, the nonlinearity $f(z,\cdot)$ is $(p-1)$-superlinear near $+\infty$. However, this superlinearity of $f(z,\cdot)$ is not formulated using the AR-condition. We recall that the AR-condition (unilateral version due to \eqref{eq4}), says that there exist $\gamma>p$ and $M>0$ such that
\begin{equation}
    0<\gamma F(z,x) \leq f(z,x) x \ \mbox{for almost all}\ z\in\Omega,\ \mbox{and all}\ x\geq M,    \tag{6a}\label{eq6a}
\end{equation}
 \begin{equation}
	0<{\rm ess\,inf}_{\Omega} F(\cdot,M).\tag{6b}\label{eq6b}
\end{equation}
\stepcounter{equation}
If we integrate \eqref{eq6a} and use \eqref{eq6b}, we obtain the weaker condition
\begin{equation}
    \begin{array}{ll}
    &c_9 x^{\gamma} \leq F(z,x)\ \mbox{for almost all}\ z\in\Omega,\ \mbox{all}\ x\geq M,\ \mbox{and some}\ c_9>0, \\
    \Rightarrow &c_9 x^{\gamma-1}\leq f(z,x)\ \mbox{for almost all}\ z\in\Omega,\ \mbox{and all}\ x\geq M.
    \end{array}
    \label{eq7}
\end{equation}

Therefore the AR-condition implies that $f(z,\cdot)$ exhibits at least $(\gamma-1)$-polynomial growth. Evidently, \eqref{eq7} implies the much weaker condition \eqref{eq5}. In this work instead of the standard AR-condition, we employ the less restrictive hypothesis $H(f)(iii)$. In this way we incorporate in our framework also $(p-1)$-superlinear terms with ``slower" growth near $+\infty$, which fail to satisfy the AR-condition. The following function satisfies hypotheses $H(f)$ but fails to satisfy the AR-condition (for the sake of simplicity we drop the $z$-dependence)
$$ f(x) = x^{p-1}\ln(1+x) \ \mbox{for all}\ x\geq 0. $$

Finally, let us fix the notation which we will use throughout this work. For $x\in\RR$ we set $x^{\pm} = \max\{ \pm x,0\}.$ Then for $u\in W^{1,p}_0(\Omega)$ we define $u^{\pm}(z) = u(z)^{\pm}$ for almost all $z\in\Omega$. It follows that
$$ u^{\pm}\in W^{1,p}_0(\Omega),\ u=u^{+}-u^{-},\ |u|=u^{+} + u^{-}. $$

If $u,v\in W^{1,p}_0(\Omega)$ and $u\leq v$, then we define.
$$
\begin{array}{ll}
&[u,v] = \{ y\in W^{1,p}_0(\Omega): u(z) \leq y(z) \leq v(z) \ \mbox{for almost all}\ z\in\Omega \}, \\
&[u) = \{ y\in W^{1,p}_0(\Omega): u(z)\leq y(z)\ \mbox{for almost all}\ z\in\Omega \}.
\end{array}
$$

Also, by ${\rm int}_{c_0^1(\overline{\Omega})}[u,v]$ we denote the interior in the $C_0^1(\overline{\Omega})$-norm topology of the set $[u,v]\cap C_0^1(\overline{\Omega})$.

By $A:W^{1,p}_0(\Omega)\rightarrow W^{-1,p'}(\Omega)=W^{1,p}_0(\Omega)^*\ (\frac{1}{p} + \frac{1}{p'} = 1)$ we denote the nonlinear operator defined by
$$\langle A(u),h\rangle =\int_{\Omega} (a(Du),Dh)_{\RR^{N}} dz \ \mbox{for all}\ u,h\in W^{1,p}_0(\Omega). $$

We know (see Gasinski \& Papageorgiou \cite{4}), that $A(\cdot)$ is continuous, strictly monotone (hence maximal monotone, too) and of type $(S)_+$, that is,
$$
\begin{array}{ll}
``\mbox{if}\ u_n\stackrel{w}{\rightarrow} u \ \mbox{in}\ W^{1,p}_0(\Omega) \ \mbox{and}\ \limsup_{n\to\infty}\langle A(u_n),u_n -u\rangle \leq 0, \\
\mbox{then}\ u_n\rightarrow u \ \mbox{in}\ W^{1,p}_0(\Omega)."
\end{array}
$$

We introduce the following two sets related to problem \eqref{eqp}:
$$
\begin{array}{ll}
    \mathcal{L} = \{ \lambda >0 :\ \mbox{problem \eqref{eqp} admits a positive solution} \},\\
    S_{\lambda} =\ \mbox{the set of positive solutions for problem \eqref{eqp}.}
\end{array}
$$

We let $\lambda^* = \sup\mathcal{L}. $

\section{Positive solutions}\label{sec3}

We start by considering the following purely singular problem:
\begin{equation}
    -{\rm div}\, a(Du(z))) + \xi^+(z)u(z)^{p-1} = \vartheta(u(z))\ \mbox{in}\ \Omega,\ u|_{\partial\Omega} = 0, \   u>0.
    \label{eq8}
\end{equation}

From Papageorgiou, R\u{a}dulescu \& Repov\v{s} \cite[Proposition 10]{16}, we have the following property.
\begin{prop}\label{prop5}
If hypotheses $H(a),H(\xi),H(\vartheta)$ hold, then problem \eqref{eq8} admits a unique positive solution $v\in{\rm int}\,C_+$.
\end{prop}

Let $\beta > ||\xi||_\infty$. Then hypotheses $H(f)(i),(iv)$ and since $1<q<p<r$, imply that we can find $c_{10},c_{11}>0$ such that
\begin{equation}
    \lambda x^{q-1} + f(z,x) \leq \lambda c_{10} x^{q-1} + c_{11}x^{r-1} -\beta x^{p-1}\ \mbox{for almost all}\ z\in\Omega,\ \mbox{and all}\ x\geq 0.
    \label{eq9}
\end{equation}

Let $k_\lambda(x) = \lambda c_{10} x^{q-1} + c_{11}x^{r-1} -\beta x^{p-1}$ for all $x\geq 0$. With $v\in {\rm int}\,C_+$ from Proposition \ref{prop5}, we consider the following auxiliary Dirichlet problem:
\begin{equation}
\left\{
        \begin{array}{ll}
            -{\rm div}\, a(Du(z)) + \xi(z) u(z)^{p-1} = \vartheta(v(z)) + k_\lambda(u(z))\ \mbox{in}\ \Omega\\
            u|_{\partial\Omega} = 0,\ u>0.
        \end{array}
    \right\}\tag*{$(10)_\lambda$}\label{eq10}
\end{equation}
 \addtocounter{equation}{+1}

For this problem we prove the following result.
\begin{prop}\label{prop6}
    If hypotheses $H(a),H(\xi),H(\vartheta)$ hold, then for all small enough $\lambda>0$ problem \ref{eq10} has a smallest positive solution
    $$ \overline{u}_\lambda \in {\rm int}\,C_+. $$
\end{prop}
\begin{proof}
Recall that $v\in {\rm int}\,C_+$ (see Proposition \ref{prop5}). Hence $v\in {\rm int}\,K_+$ (see \eqref{eq1}). For $s>N$ we consider the function $\hat{u}^{1/s}_1\in K_+$. According to Proposition \ref{prop1}, we can find $\mu>0$ such that
\begin{equation}
    \begin{array}{ll}
        &\hat{u}^{1/s}_1\leq \mu v,\\
        \Rightarrow &v^{-\gamma}\leq \mu^{\gamma} \hat{u}_1^{-\gamma/s} .
    \end{array}
    \label{eq11}
\end{equation}

From the Lemma in Lazer \& McKenna \cite[p. 726]{9}, we have
\begin{equation}
    \begin{array}{ll}
        &\hat{u}_1^{-\gamma/s} \in L^s (\Omega),\\
        \Rightarrow&v^{-\gamma}\in L^s(\Omega)\ (\mbox{see}\ \eqref{eq11}).
    \end{array}
    \label{eq12}
\end{equation}

Hypotheses $H(\vartheta)$ imply that we can find $c_{12}>0$ and $\delta>0$ such that
\begin{equation}
    0\leq\vartheta(x)\leq c_{12}x^{-\gamma} \ \mbox{for all}\ 0\leq x\leq\delta\ \mbox{and}\ 0\leq\vartheta(x)  \leq\vartheta(\delta)\ \mbox{for all}\ x>\delta.
    \label{eq13}
\end{equation}

It follows from \eqref{eq12}, \eqref{eq13} that
$$ \vartheta (v(\cdot))\in L^s(\Omega)\ (s>N).$$

Let $\hat{k}_\lambda(x) = \lambda c_{10} x^{q-1} + c_{11}x^{r-1}$ for all $x\geq 0$ and set $\hat{K}_\lambda(x) = \int_{0}^{x} \hat{k}_\lambda(s)ds$. We consider the $C^1$- functional $\psi_\lambda: W^{1,p}_0(\Omega)\rightarrow \RR$ defined by
\begin{eqnarray}
    &&\psi_\lambda(u) = \int_{\Omega} G(Du)dz+\frac{1}{p}\int_{\Omega}[\xi(z)+\beta]|u|^p dz-\int_{\Omega}\hat{K}_\lambda(u^+)dz - \int_{\Omega}\vartheta(v) u^+ dz \nonumber\\
    &&\mbox{for all}\ u\in W^{1,p}_0(\Omega)\nonumber\\
    &&\geq\frac{c_1}{p(p-1)}||Du||_p^p +\frac{1}{p}\int_{\Omega} [\xi(z)+\beta]|u|^p dz - \frac{\lambda c_{10}}{q}||u||_q^q - \frac{c_{11}}{r}||u||_r^r - \int_{\Omega}\vartheta(v)|u|dz  \nonumber\\
    &&\mbox{ (see Corollary \ref{cor3})}\nonumber\\
    &&\geq c_{12}||u||^p -c_{13}[\lambda||u||+||u||^r]\nonumber\\
    &&\mbox{for some}\ c_{12},c_{13}>0\ \mbox{and all}\ 0<\lambda\leq 1\ (\mbox{recall that}\ \beta>||\xi||_\infty \ \mbox{and}\ 1<q<r)\nonumber\\
    &&=[c_{12} - c_{13}(\lambda||u||^{1-p}+||u||^{r-p})]||u||^p.
    \label{eq14}
\end{eqnarray}

We introduce the function $\Im_\lambda(t)=\lambda t^{1-p} + t^{r-p},t>0. $ Evidently, $\Im_\lambda \in C^{1}(0,+\infty)$ and since $1<p<r$, we see that
$$ \Im_\lambda(t)\rightarrow+\infty \ \mbox{as}\ t\rightarrow 0^+ \ \mbox{and as}\ t\rightarrow+\infty.$$
So, we can find $t_0>0$ such that
$$
\begin{array}{ll}
    &\Im_\lambda(t_0)=\inf\{\Im_\lambda(t):t>0\},\\
    \Rightarrow&\Im_\lambda'(t_0) = 0,\\
    \Rightarrow&\lambda(p-1)t^{-p}_0 = (r-p) t^{r-p-1}_0\\
    \Rightarrow&t_0 = \left[ \frac{\lambda(p-1)}{r-p} \right]^{\frac{1}{r-1}}.
\end{array}
$$

Since $\frac{p-1}{r-1}<1$, it follows that
$$ \Im_\lambda(t_0)\rightarrow 0\ \mbox{as}\ \lambda\rightarrow 0^+ .$$

So, we can find $\lambda_0\in(0,1]$ such that
$$ \Im_\lambda(t_0)\leq \frac{c_{12}}{c_{13}} \ \mbox{for all}\ \lambda\in(0,\lambda_0]. $$

For $ \rho=t_0 $, we see from \eqref{eq14} that
\begin{equation}
    \psi_\lambda|_{\partial\overline{B}_\rho}>0,
    \label{eq15}
\end{equation}
where $\overline{B}_{\rho} = \{u\in W^{1,p}_0(\Omega):||u||\leq\rho \}$ and $\partial\overline{B}_\rho = \{u\in W^{1,p}_0 (\Omega):||u||=\rho\}$.

We fix $\lambda\in(0,\lambda_0]$. Hypothesis $H(a)(iv)$ implies that we can find $c^*_0 > c^*$ and $\delta>0$ such that
$$ G(y)\leq \frac{c_0^*}{\tau}|y|^{\tau}\ \mbox{for all}\ |y|\leq\delta. $$

Let $u\in {\rm int}\, C_+$ and choose small enough $t\in(0,1)$ such that
$$ t|Du(z)|\leq\delta\ \mbox{for all}\ z\in\overline{\Omega}.  $$

Then we have
$$ \psi_\lambda(tu)\leq\frac{t^{\tau}c_0^*}{\tau}||Du||_\tau^\tau + \frac{t^p}{p}\int_{\Omega}[\xi(z)+\beta]|u|^p dz- \frac{\lambda t^q}{q}||u||_q^q\,. $$

Since $q<\tau\leq p$, choosing $t\in(0,1)$\ even smaller if necessary, we have
\begin{eqnarray}
    &\psi_\lambda(tu)<0,\nonumber\\
    \Rightarrow&\inf_{\overline{B}_\rho}\psi_\lambda<0.
    \label{eq16}
\end{eqnarray}

The functional $\psi_\lambda(\cdot)$ is sequentially weakly lower semicontinuous and by the Eberlein-Smulian theorem and the reflexivity of $W^{1,p}_0(\Omega)$, the set $\overline{B}_\rho$ is sequentially weakly compact. So, by the Weierstrass-Tonelli theorem, we can find $\overline{u}\in W^{1,p}_0(\Omega)$ such that
\begin{equation}
    \psi_\lambda(\overline{u}) = \inf\{\psi_\lambda(u): u\in W^{1,p}_0(\Omega)\}\ \ (\lambda\in(0,\lambda_0]).
    \label{eq17}
\end{equation}

From \eqref{eq15}, \eqref{eq16} and \eqref{eq17} it follows that
\begin{eqnarray}
    &&0<||\overline{u}||<\rho,\nonumber\\
    &&\Rightarrow \psi'_\lambda (\overline{u}) = 0\ (\mbox{see}\ \eqref{eq17}),\nonumber\\
    &&\Rightarrow \langle A(\overline{u}),h\rangle + \int_{\Omega}[\xi(z)+\beta]|\overline{u}|^{p-2}\overline{u}hdz = \int_{\Omega} [\vartheta(v)+\hat{k}_\lambda(\overline{u}^+)] hdz \label{eq18}\\
    &&\mbox{for all}\ h\in W^{1,p}_0(\Omega).\nonumber
\end{eqnarray}

In \eqref{eq18} we choose $h=-\overline{u}^{-}\in W^{1,p}_0(\Omega).$ Using Lemma \ref{lem2}$(c)$ and since $\beta>||\xi||_\infty$ we obtain
\begin{eqnarray*}
    &&c_{14}||\overline{u}^{-}||^p \leq 0\ \mbox{for some}\ c_{14}>0,\nonumber\\
    &&\Rightarrow \overline{u}\geq 0,\ \overline{u}\neq 0.\nonumber
\end{eqnarray*}

Then from \eqref{eq18} we have
\begin{eqnarray}
&&-{\rm div}\, a(D\overline{u}(z)) + \xi(z)\overline{u}(z)^{p-1} = \vartheta(v(z)) + k_\lambda(\overline{u}(z)) \ \mbox{for almost all}\ z\in\Omega,  \label{eq19}\\
&&\Rightarrow \overline{u}\in W^{1,p}_0 (\Omega)\ \mbox{is a positive of problem \ref{eq10} for}\ \lambda\in(0,\lambda_0].\nonumber
\end{eqnarray}

From \eqref{eq19} and Theorem 7.1 of Ladyzhenskaya \& Uraltseva \cite[p. 286]{8}, we have $\overline{u}\in L^{\infty}(\Omega)$. Hence $k_\lambda(\overline{u}(\cdot))\in L^\infty(\Omega)$. Recall that $\vartheta(v(\cdot))\in L^s(\Omega)$ with $s>N$. From Theorem 9.15 of Gilbarg \& Trudinger \cite[p. 241]{6}, we know that there exists a unique solution $y_0\in W^{2,s}(\Omega)$ to the following linear Dirichlet problem
$$ -\Delta y(z) = \vartheta(v(z))\ \mbox{in}\ \Omega ,\ y|_{\partial\Omega} =0 .$$

By the Sobolev embedding theorem, we have
$$ W^{2,s}(\Omega) \hookrightarrow C^{1,\alpha}(\overline{\Omega})\ \mbox{with}\	\alpha = 1-\frac{N}{s}>0. $$

Let $\eta_0(z)=Dy_0(z).$ Then $\eta_0\in C^{\alpha}(\overline{\Omega},\RR^{N})$ and we have
$$ -{\rm div}\,(a(D\overline{u}(z))-\eta_0(z)) + \xi(z)\overline{u}(z)^{p-1} = k_\lambda(\overline{u}(z))\ \mbox{for almost all}\ z\in\Omega. $$

The regularity theory of Lieberman \cite{10} implies that $\overline{u}\in C_+ \backslash \{ 0 \}$. Moreover, from \eqref{eq19} we have
\begin{eqnarray*}
	&&{\rm div}\, a(D\overline{u}(z))\leq ||\xi||_\infty \overline{u}(z)^{p-1}\ \mbox{for almost all}\ z\in\Omega,\nonumber\\
	\Rightarrow&&\overline{u}\in{\rm int}\,C_+\nonumber
\end{eqnarray*}
(from the nonlinear maximum principle, see Pucci \& Serrin \cite[pp. 111,120]{21}).

Let $\hat{S_\lambda}$ denote the set of positive solutions of problem \ref{eq10}. We have just seen that $\emptyset\neq \hat{S_\lambda}\subseteq {\rm int}\,C_+$ for $\lambda\in(0,\lambda_0].$ Moreover, from Papageorgiou, R\u{a}dulescu \& Repov\v{s} \cite[Proposition 18]{16}, we know that $\hat{S_\lambda}$ is downward directed (that is, if $u_1,u_2,\in \hat{S_\lambda}$, then we can find $u\in\hat{S_\lambda}$ such that $u\leq u_1,u\leq u_2)$. So, by Lemma 3.10 of Hu \& Papageorgiou \cite[p. 178]{7}, we can find a decreasing sequence $\{ \overline{u}_n \}_{n\geq1} \subseteq \hat{S_\lambda}$ such that
$$ \inf\hat{S_\lambda} = \inf_{n\geq1}\overline{u}_n. $$

For every $n\in\NN$ we have
\begin{equation}
	\langle A(\overline{u}_n),h\rangle + \int_\Omega\xi(z)\overline{u}_n^{p-1} hdz = \int_{\Omega}[\vartheta(v)+ k_\lambda(\overline{u}_n)] hdz\ \mbox{for all}\ h\in W^{1,p}_0 (\Omega).
	\label{eq20}
\end{equation}

Choosing $h = \overline{u}_n \in W^{1,p}_0$ and since $0\leq\overline{u}_n\leq\overline{u}_1$ for all $n\in\NN$, using Lemma 2(c), we see that $\{\overline{u}_n\}_{n\geq 1} \subseteq W^{1,p}_0(\Omega)$ is bounded. So, we have
\begin{equation}
	\overline{u}_n \stackrel{w}{\rightarrow} \overline{u}_\lambda\ \mbox{in}\ W^{1,p}_0(\Omega).
	\label{eq21}
\end{equation}

Next, in \eqref{eq20} we choose $h=\overline{u}_n-\overline{u}\in W^{1,p}_0(\Omega)$, pass to the limit as $n\rightarrow\infty$ and use \eqref{eq21}. Then
\begin{eqnarray}
	&&\lim_{n\to\infty} \langle A(\overline{u}_n),\overline{u}_n-\overline{u}_\lambda\rangle = 0,\nonumber\\
	\Rightarrow&& \overline{u}_n\rightarrow \overline{u}_\lambda\ \mbox{in}\ W^{1,p}_0(\Omega),\ \overline{u}_\lambda \geq 0 	\label{eq22} \\
	&&\mbox{(recall that }\ A(\cdot)\ \mbox{is of type}\ (S)_+,\ \mbox{see Section \ref{sec2}}).\nonumber
\end{eqnarray}

We pass to the limit as $n\rightarrow\infty$ in \eqref{eq20} and use \eqref{eq22}. Then
\begin{eqnarray*}
	&&\langle A(\overline{u}_\lambda), h\rangle + \int_{\Omega}\xi(z)\overline{u}_\lambda^{p-1} hdz = \int_{\Omega}[\vartheta(v)+ k_\lambda(\overline{u}_\lambda) ] hdz\nonumber\\
	&&\mbox{for all}\ h\in W^{1,p}_0 (\Omega),\nonumber\\
	\Rightarrow&&\overline{u}_\lambda\ \mbox{is a nonnegative solution of \ref{eq10}}.\nonumber
\end{eqnarray*}

Note that for all $n\in\NN$, we have
\begin{equation*}
	\begin{array}{ll}
		-{\rm div}\, a(D\overline{u}_n(z)) + \xi^+(z)\overline{u}_n(z)^{p-1} & \geq \vartheta (v(z)) + k_\lambda (\overline{u}_n (z)) \\
		&\geq \vartheta(v(z)) = -{\rm div}\, a(Dv(z)) + \xi^{+}(z) v (z)^{p-1}\\ &\mbox{for almost all}\ z\in \Omega,\\
		\Rightarrow v\leq \overline{u}_n\ \mbox{for all}\ n\in\NN
	\end{array}
\end{equation*}
(by the weak comparison principle, see Damascelli \cite[Theorem 1.2]{3})
\begin{equation}
	\Rightarrow v\leq \overline{u}_\lambda \ \mbox{(see \eqref{eq22}), hence}\ \overline{u}_\lambda\neq 0.
	\label{eq23}
\end{equation}

Therefore $\overline{u}_\lambda \in \hat{S_\lambda}\subseteq {\rm int}\,C_+$ and $\overline{u}_\lambda = \inf\hat{S_\lambda}$.
\end{proof}

We will use $\bar{u}_\lambda\in{\rm int}\,C_+$ from Proposition \ref{prop6} to show the nonemptiness of $\mathcal{L}$.
\begin{prop}\label{prop7}
	If hypotheses $H(a),H(\xi),H(\vartheta),H(f)$ hold, then $\mathcal{L}\neq\emptyset$ and $S_\lambda\subseteq{\rm int}\,C_+$.
\end{prop}
\begin{proof}
	From (\ref{eq9}) we have
	\begin{equation}\label{eq24}
		\lambda x^{q-1}+f(z,x)\leq k_\lambda(x)\ \mbox{for almost all}\ z\in\Omega,\mbox{ and all }x\geq 0,\,  \lambda>0.
	\end{equation}
	
	For $\lambda\in\left(0,\lambda_0\right]$ we have
	\begin{eqnarray}\label{eq25}
		-{\rm div}\,a(D\bar{u}_\lambda(z))+\xi(z)\bar{u}_\lambda(z)^{p-1}&=&\vartheta(v(z))+k_\lambda(\bar{u}_\lambda(z))\nonumber\\
		&&(\mbox{see Proposition \ref{prop6}})\nonumber\\
		&\geq&\vartheta(\bar{u}_\lambda(z))+k_\lambda(\bar{u}(z))\nonumber\\
		&&(\mbox{see (\ref{eq23}) and hypothesis}\ H(\vartheta)(ii))\nonumber\\
		&\geq&\vartheta(\bar{u}_\lambda(z))+\lambda\bar{u}_\lambda(z)^{q-1}+f(z,\bar{u}_\lambda(z))\\
		&&\mbox{for almost all}\ z\in\Omega\ (\mbox{see (\ref{eq24})}).\nonumber
	\end{eqnarray}
	
	With $\beta>||\xi||_\infty$ and $\lambda\in\left(0,\lambda_0\right]$, we consider the following truncation-perturbation of the reaction in problem \eqref{eqp}:
	\begin{eqnarray}\label{eq26}
		&&\gamma_\lambda(z,x)=\left\{\begin{array}{ll}
			\vartheta(v(z))+\lambda v(z)^{q-1}+f(z,v(z))+\beta v(z)^{p-1}&\mbox{if}\ x<v(z)\\
			\vartheta(x)+\lambda x^{q-1}+f(z,x)+\beta x^{p-1}&\mbox{if}\ v(z)\leq x\leq \bar{u}_\lambda(z)\\
			\vartheta(\bar{u}_\lambda(z))+\lambda\bar{u}_\lambda(z)^{q-1}+f(z,\bar{u}_\lambda(z))+\beta\bar{u}_\lambda(z)^{p-1}&\mbox{if}\ \bar{u}_\lambda(z)<x.
		\end{array}\right.
	\end{eqnarray}
	
	This is a Carath\'eodory function. We set $\Gamma_\lambda(z,x)=\int^x_0\gamma_\lambda(z,s)ds$ and consider the functional $\hat{\sigma}_\lambda:W^{1,p}_0(\Omega)\rightarrow\RR$ defined by
	\begin{eqnarray*}
		&&\hat{\sigma}_\lambda(u)=\int_\Omega G(Du)dz+\frac{1}{p}\int_\Omega[\xi(z)+\beta]|u|^pdz-\int_\Omega\Gamma_\lambda(z,u)dz\\
		&&\mbox{for all}\ u\in W^{1,p}_0(\Omega).
	\end{eqnarray*}
	
	Using Proposition 3 of Papageorgiou \& Smyrlis \cite{17}, we see that $\hat{\sigma}_\lambda\in C^1(W^{1,p}_0(\Omega))$. Also, from (\ref{eq26}), Corollary \ref{cor3} and since $\beta>||\xi||_\infty$, we see that $\hat{\sigma}_\lambda(\cdot)$ is coercive. In addition, it is sequentially weakly lower semicontinuous. So, we can find $u_\lambda\in W^{1,p}_0(\Omega)$ such that
	\begin{eqnarray}\label{eq27}
		&&\hat{\sigma}_\lambda(u_\lambda)=\inf\{\hat{\sigma}(u):u\in W^{1,p}_0(\Omega)\},\nonumber\\
		&\Rightarrow&\hat{\sigma}'_\lambda(u_\lambda)=0,\nonumber\\
		&\Rightarrow&\left\langle A(u_\lambda,h)\right\rangle+\int_\Omega[\xi(z)+\beta]|u_\lambda|^{p-2}u_\lambda hdz=\int_\Omega \gamma_\lambda(z,u_\lambda)hdz\\
		&&\mbox{for all}\ h\in W^{1,p}_0(\Omega).\nonumber
	\end{eqnarray}
	
	In (\ref{eq27}) first we choose $h=(u_\lambda-\bar{u}_\lambda)^+\in W^{1,p}_0(\Omega)$. Then we have
	\begin{eqnarray*}
		&&\left\langle A(u_\lambda),(u_\lambda-\bar{u}_\lambda)^+\right\rangle+\int_\Omega[\xi(z)+\beta]u_\lambda^{p-1}(u_\lambda-\bar{u}_\lambda)^+dz\\
		&=&\int_\Omega[\vartheta(\bar{u}_\lambda)+\lambda\bar{u}_\lambda^{q-1}+f(z,\bar{u}_\lambda)+\beta\bar{u}_\lambda^{p-1}](u_\lambda-\bar{u}_\lambda)^+dz\ (\mbox{see (\ref{eq26})})\\
		&\leq&\left\langle A(\bar{u}_\lambda),(u_\lambda-\bar{u}_\lambda)^+\right\rangle+\int_\Omega[\xi(z)+\beta]\bar{u}_\lambda^{p-1}(u_\lambda-\bar{u}_\lambda)^+dz\ (\mbox{see (\ref{eq25})}),\\
		&\Rightarrow&u_\lambda\leq\bar{u}_\lambda\ (\mbox{since}\ \beta>||\xi||_\infty).
	\end{eqnarray*}
	
	Next, in (\ref{eq27}) we choose $h=(v-u_\lambda)^+\in W^{1,p}_0(\Omega)$. Then we have
	\begin{eqnarray*}
		&&\left\langle A(u_\lambda),(v-u_\lambda)^+\right\rangle+\int_\Omega[\xi(z)+\beta]|u_\lambda|^{p-2}u_\lambda(v-u_\lambda)^+dz\\
		&=&\int_\Omega[\vartheta(v)+\lambda v^{q-1}+f(z,v)+\beta v^{p-1}](v-u_\lambda)^+dz\ (\mbox{see (\ref{eq26})})\\
		&\geq&\int_\Omega[\vartheta(v)+\beta v^{p-1}](v-u_\lambda)^+dz\ (\mbox{since}\ f\geq 0)\\
		&=&\left\langle A(v),(v-u_\lambda)^+\right\rangle+\int_\Omega[\xi(z)+\beta]v^{p-1}(v-u_\lambda)^+dz\ (\mbox{see Proposition \ref{prop5}}),\\
		&\Rightarrow&v\leq u_\lambda.
	\end{eqnarray*}
	
	So, we have proved that
	\begin{equation}\label{eq28}
		u_\lambda\in[v,\bar{u}_\lambda]\quad (\lambda\in\left(0,\lambda_0\right]).
	\end{equation}
	
It follows	from (\ref{eq26}), (\ref{eq27}) and (\ref{eq28}) that
	\begin{eqnarray*}
		&&-{\rm div}\,a(Du_\lambda(z))+\xi(z)u_\lambda(z)^{p-1}=\vartheta(u_\lambda(z))+\lambda u_\lambda(z)^{q-1}+f(z,u_\lambda(z))\\
		&&\mbox{for almost all}\ z\in\Omega.
	\end{eqnarray*}
	
	Note that $\vartheta_\lambda(u_\lambda)\leq\vartheta(v)$ (see (\ref{eq28}) and hypothesis $H(\vartheta)(ii)$) and $\vartheta(v)\in L^s(\Omega)$. So, as before (see the proof of Proposition \ref{prop6}), we infer that
	$$u_\lambda\in{\rm int}\,C_+\,.$$
	
	Therefore we have seen that
	\begin{eqnarray*}
		&&\left(0,\lambda_0\right]\subseteq\mathcal{L},\ \mbox{hence}\ \mathcal{L}\neq\emptyset\\
		\mbox{and}&&S_\lambda\subseteq{\rm int}\,C_+.
	\end{eqnarray*}
The proof is now complete.
\end{proof}

For $\eta>0$, let $\tilde{u}_{\eta}\in{\rm int}\,C_+$ be the unique solution of the following Dirichlet problem
$$-{\rm div}\,a(Du(z))+\xi^+(z)u(z)^{p-1}=\eta\ \mbox{in}\ \Omega,\ u|_{\partial\Omega}=0.$$

By Proposition 9 of Papageorgiou, R\u{a}dulescu \& Repov\v{s} \cite{16}, we see that given $u\in S_\lambda\subseteq{\rm int}\,C_+$ (that is, $\lambda\in\mathcal{L}$), we can find small $\eta>0$ such that
\begin{equation}\label{eq29}
	\tilde{u}_\eta\leq u\ \mbox{and}\ \eta\leq\vartheta(\tilde{u}_\eta).
\end{equation}

We will use this to obtain a lower bound for the elements of $S_\lambda$.
\begin{prop}\label{prop8}
	If hypotheses $H(a),H(\xi),H(\vartheta),H(f)$ hold and $\lambda\in\mathcal{L}$, then $v\leq u$ for all $u\in S_\lambda$.
\end{prop}
\begin{proof}
	Let $u\in S_\lambda\subseteq{\rm int}\,C_+$. Then on account of (\ref{eq29}) we can define the following Carath\'eodory function
	\begin{equation}\label{eq30}
		e(z,x)=\left\{\begin{array}{ll}
			\vartheta(\tilde{u}_\eta(z))&\mbox{if}\ x<\tilde{u}_\eta(z)\\
			\vartheta(x)&\mbox{if}\ \tilde{u}_\eta(z)\leq x\leq u(z)\\
			\vartheta(u(z))&\mbox{if}\ u(z)<x.
		\end{array}\right.
	\end{equation}
	
	We set $E(z,x)=\int^x_0 e(z,s)ds$ and consider the functional $\mu:W^{1,p}_0(\Omega)\rightarrow\RR$ defined by
	$$\mu(u)=\int_\Omega G(Du)dz+\frac{1}{p}\int_\Omega\xi^+(z)|u|^pdz-\int_\Omega E(z,u)dz\ \mbox{for all}\ u\in W^{1,p}_0(\Omega).$$
	
	As before, Proposition 3 of Papageorgiou \& Smyrlis \cite{17} implies that $\mu\in C^1(W^{1,p}_0(\Omega))$. The coercivity of $\mu(\cdot)$ (see (\ref{eq30})) and the sequential weak lower semicontinuity guarantee the existence of $\tilde{v}\in W^{1,p}_0(\Omega)$ such that
	\begin{eqnarray}\label{eq31}
		&&\mu(\tilde{v})=\inf\{\mu(u):u\in W^{1,p}_0(\Omega)\},\nonumber\\
		&\Rightarrow&\mu'(\tilde{v})=0,\nonumber\\
		&\Rightarrow&\left\langle A(\tilde{v}),h\right\rangle+\int_\Omega\xi^+(z)|\tilde{v}|^{p-2}\tilde{v}hdz=\int_\Omega e(z,\tilde{v})hdz\ \mbox{for all}\ h\in W^{1,p}_0(\Omega).
	\end{eqnarray}
	
	In (\ref{eq31}) we choose $h=(\tilde{v}-u)^+\in W^{1,p}_0(\Omega)$. Then we have
	\begin{eqnarray*}
		&&\left\langle A(\tilde{v}),(\tilde{v}-u)^+\right\rangle+\int_\Omega\xi^+(z)\tilde{v}^{p-1}(\tilde{v}-u)^+dz\\
		&=&\int_\Omega\vartheta(u)(\tilde{v}-u)^+dz\ (\mbox{see (\ref{eq30})})\\
		&\leq&\int_\Omega[\vartheta(u)+\lambda u^{q-1}+f(z,u)](\tilde{v}-u)^+dz\ (\mbox{since}\ u\in{\rm int}\,C_+,\ f\geq 0)\\
		&\leq&\left\langle A(u),(\tilde{v}-u)^+\right\rangle+\int_\Omega\xi^+(z)u^{p-1}(\tilde{v}-u)^+dz\ (\mbox{since}\ u\in S_\lambda),\\
		\Rightarrow&&\tilde{v}\leq u.
	\end{eqnarray*}
	
	Similarly, if in (\ref{eq31}) we choose $h=(\tilde{u}_\eta-\tilde{v})^+\in W^{1,p}_0(\Omega)$, then we have
	\begin{eqnarray*}
		&&\left\langle A(\tilde{v}),(\tilde{u}_\eta-\tilde{v})^+\right\rangle+\int_\Omega\xi^+(z)|\tilde{v}|^{p-2}\tilde{v}(\tilde{u}_\eta-\tilde{v})^+dz\\
		&=&\int_\Omega\vartheta(\tilde{u}_\eta)(\tilde{u}_\eta-\tilde{v})^+dz\ (\mbox{see (\ref{eq30})})\\
		&\geq&\int_\Omega\eta(\tilde{u}_\eta-\tilde{v})^+dz\ (\mbox{see (\ref{eq29})})\\
		&=&\left\langle A(\tilde{u}_\eta),(\tilde{u}_\eta-v)^+\right\rangle+\int_\Omega\xi^+(z)\tilde{u}_{\eta}^{p-1}(\tilde{u}_\eta-v)^+dz,\\
		\Rightarrow&&\tilde{u}_\eta\leq\tilde{v}.
	\end{eqnarray*}
	
	So, we have proved that
	\begin{equation}\label{eq32}
		\tilde{v}\in[\tilde{u}_\eta,u].
	\end{equation}
	
	It follows from (\ref{eq30}), (\ref{eq31}), (\ref{eq32}) that $\tilde{v}$ is a positive solution of (\ref{eq18}). Then on account of Proposition \ref{prop5}, we have
	\begin{eqnarray*}
		&&\tilde{v}=v\in{\rm int}\,C_+,\\
		&\Rightarrow&v\leq u\ \mbox{for all}\ u\in S_\lambda\ (\mbox{see (\ref{eq32})}).
	\end{eqnarray*}
The proof is now complete.
\end{proof}

Next, we show a structural property of the set $\mathcal{L}$, namely that $\mathcal{L}$ is an interval. Moreover, we establish a kind of strong monotonicity property for the solution set $S_\lambda$.
\begin{prop}\label{prop9}
	If hypotheses $H(a),H(\xi),H(\vartheta),H(f)$ hold, $\lambda\in\mathcal{L},0<\mu<\lambda$ and $u_\lambda\in S_\lambda\subseteq{\rm int}\,C_+$, then $\mu\in\mathcal{L}$ and there exists $u_\mu\in S_\mu\subseteq{\rm int}\,C_+$ such that $u_\lambda-u_\mu\in{\rm int}\,C_+.$
\end{prop}
\begin{proof}
	From Proposition \ref{prop8} we know that $v\leq u_\lambda$. Then with $\beta>||\xi||_\infty$ we can define the following truncation-perturbation of the reaction in problem ($P_\mu$):
	\begin{equation}\label{eq33}
		e_\mu(z,x)=\left\{\begin{array}{ll}
			\vartheta(v(z))+\mu v(z)^{q-1}+f(z,v(z))+\beta v(z)^{p-1}&\mbox{if}\ x<v(z)\\
			\vartheta(x)+\mu x^{q-1}+f(z,x)+\beta x^{p-1}&\mbox{if}\ v(z)\leq x\leq u_\lambda(z)\\
			\vartheta(u_\lambda(z))+\mu u_\lambda(z)^{q-1}+f(z,u_\lambda(z))+\beta u_\lambda(z)^{p-1}&\mbox{if}\ u_\lambda(z)<x.
		\end{array}\right.
	\end{equation}
	
	Evidently, $e_\mu(z,x)$ is a Carath\'eodory function. We set $E_\mu(z,x)=\int^x_0e_\mu(z,s)ds$ and consider the $C^1$-functional $\hat{\psi}_\mu:W^{1,p}_0(\Omega)\rightarrow\RR$ defined by
	$$\hat{\psi}_\mu(u)=\int_\Omega G(Du)dz+\frac{1}{p}\int_\Omega[\xi(z)+\beta]|u|^pdz-\int_\Omega E_\mu(z,u)dz\ \mbox{for all}\ u\in W^{1,p}_0(\Omega).$$
	
	Clearly, $\hat{\psi}_\mu(\cdot)$ is coercive (see (\ref{eq33}) and recall that $\beta>||\xi||_\infty$). It is also sequentially weakly lower semicontinuous. So, we can find $u_\mu\in W^{1,p}_0(\Omega)$ such that
	\begin{eqnarray}\label{eq34}
		&&\hat{\psi}_\mu(u_\mu)=\inf\{\hat{\psi}_\mu(u):u\in W^{1,p}_0(\Omega)\},\nonumber\\
		&\Rightarrow&\hat{\psi}'_\mu(u_\mu)=0,\nonumber\\
		&\Rightarrow&\left\langle A(u_\mu),h\right\rangle+\int_\Omega[\xi(z)+\beta]|u_\mu|^{p-2}u_\mu hdz=\int_\Omega e_\mu(z,u_\mu)hdz\ \mbox{for all}\ h\in W^{1,p}_0(\Omega).
	\end{eqnarray}
	
	In (\ref{eq34}) we first use $h=(u_\mu-u_\lambda)^+\in W^{1,p}_0(\Omega)$. Then
	\begin{eqnarray*}
		&&\left\langle A(u_\mu),(u_\mu-u_\lambda)^+\right\rangle+\int_\Omega[\xi(z)+\beta]u_\mu^{p-1}(u_\mu-u_\lambda)^+dz\\
		&=&\int_\Omega[\vartheta(u_\lambda)+\mu u_\lambda^{q-1}+f(z,u_\lambda)+\beta u_\lambda^{p-1}](u_\mu-u_\lambda)^+dz\ (\mbox{see (\ref{eq33})})\\
		&\leq&\int_\Omega[\vartheta(u_\lambda)+\lambda u_\lambda^{q-1}+f(z,u_\lambda)+\beta u_\lambda^{p-1}](u_\mu-u_\lambda)^+dz\ (\mbox{since }\lambda>\mu)\\
		&=&\left\langle A(u_\lambda),(u_\mu-u_\lambda)^+\right\rangle+\int_\Omega[\xi(z)+\beta]u_\lambda^{p-1}(u_\mu-u_\lambda)^+dz\ (\mbox{since }u_\lambda\in S_\lambda),\\
		\Rightarrow&&u_\mu\leq u_\lambda\ (\mbox{recall that}\ \beta>||\xi||_\infty).
	\end{eqnarray*}
	
	Next, in (\ref{eq34}) we use $h=(v-u_\mu)^+\in W^{1,p}_0(\Omega)$. Then from Proposition \ref{prop5} and since $f\geq 0$, we obtain
	$$v\leq u_\mu\,.$$
	
	We have proved that
	\begin{equation}\label{eq35}
		u_\mu\in[v,u_\lambda].
	\end{equation}
	
	It follows from (\ref{eq33}), (\ref{eq34}), (\ref{eq35}) that $u_\mu\in S_\mu\subseteq{\rm int}\,C_+$ and so $\mu\in\mathcal{L}$.
	
	Let $\rho=||u_\lambda||_\infty$ and let $\hat{\xi}_\rho>0$ as postulated by hypothesis $H(f)(v)$. We have
	\begin{eqnarray}\label{eq36}
		&&-{\rm div}\,a(Du_\mu)+[\xi(z)+\hat{\xi}_\rho]u_\mu^{p-1}-\vartheta(u_\mu)\nonumber\\
		&=&\mu u_\mu^{q-1}+f(z,u_\mu)+\hat{\xi}_\rho u_\mu^{p-1}\nonumber\\
		&\leq&\lambda u_\lambda^{q-1}+f(z,u_\lambda)+\hat{\xi}_\rho u_\lambda^{p-1}\ (\mbox{see hypothesis}\ H(f)(vi),\ \mbox{(\ref{eq35}) and recall that}\ \mu<\lambda)\nonumber\\
		&=&-{\rm div}\,a (Du_\lambda)+[\xi(z)+\hat{\xi}_\rho]u_\lambda^{p-1}-\vartheta(u_\lambda).
	\end{eqnarray}
	
	From (\ref{eq36}) and Proposition 4 of Papageorgiou \& Smyrlis \cite{17}, we obtain
	$$u_\lambda-u_\mu\in{\rm int}\,C_+.$$
The proof is now complete.
\end{proof}
\begin{prop}\label{prop10}
	If hypotheses $H(a),H(\xi),H(\vartheta),H(f)$ hold, then $\lambda^*<+\infty$.
\end{prop}
\begin{proof}
	Recall that by hypotheses $H(f)(ii),(iii)$, we have
	$$\lim_{x\rightarrow+\infty}\frac{f(z,x)}{x^{p-1}}=+\infty\ \mbox{uniformly for almost all}\ z\in\Omega.$$
	
	So, we can find $M>0$ such that
	\begin{equation}\label{eq37}
		f(z,x)\geq x^{p-1}\ \mbox{for almost all}\ z\in\Omega,\ \mbox{and all}\ x\geq M.
	\end{equation}
	
	Hypotheses $H(\vartheta)$ imply that we can find small $\delta\in\left(0,1\right]$ such that
	\begin{equation}\label{eq38}
		\vartheta(x)\geq\vartheta(\delta)\geq 1\geq\delta^{p-1}\geq x^{p-1}\ \mbox{for all}\ x\in\left(0,\delta\right].
	\end{equation}
	
	Finally, hypotheses $H(f)(i),(v)$ imply that we can find big $\lambda_0>0$ such that
	\begin{equation}\label{eq39}
		\lambda_0x^{q-1}+f(z,x)\geq x^{p-1}\ \mbox{for almost all}\ z\in\Omega\ \mbox{and all}\ \delta\leq x\leq M.
	\end{equation}
	
	Combining (\ref{eq37}), (\ref{eq38}), (\ref{eq39}) we have
	\begin{equation}\label{eq40}
		\vartheta(x)+\lambda_0x^{q-1}+f(z,x)\geq x^{p-1}\ \mbox{for almost all}\ z\in\Omega\ \mbox{and all}\ x\geq 0.
	\end{equation}
	
	Let $\lambda>\lambda_0$ and assume that $\lambda\in\mathcal{L}$. Then according to Proposition \ref{prop7} we can find $u_\lambda\in S_\lambda\subseteq{\rm int}\,C_+$. Let $\Omega_0\subseteq\Omega$ be an open set with $\overline{\Omega}_0\subseteq\Omega$ and $C^2$-boundary $\partial\Omega_0$. We have
	$$0<m_0=\min\limits_{\overline{\Omega}_0}u_\lambda.$$
	
	For $\epsilon>0$, let $m^{\epsilon}_{0}=m_0+\epsilon$ and with $\rho=||u_\lambda||_\infty$, let $\hat{\xi}_\rho>0$ be as postulated by hypothesis $H(f)(v)$. We can always take $\hat{\xi}_\rho>||\xi||_\infty$. We have
	\begin{eqnarray}\label{eq41}
		&&-{\rm div}\,a(Dm^\epsilon_0)+[\xi(z)+\hat{\xi}_\rho](m^\epsilon_0)^{p-1}-\vartheta(m^\epsilon_0)\nonumber\\
		&\leq&[\xi(z)+\hat{\xi}_\rho]m^{p-1}_0+\chi(\epsilon)-\vartheta(m_0)\nonumber\\
		&&\mbox{with}\ \chi(\epsilon)\rightarrow 0^+\ \mbox{as}\ \epsilon\rightarrow 0^+\ (\mbox{see hypotheses}\ H(\vartheta))\nonumber\\
		&<&[\xi(z)+\hat{\xi}_\rho]u^{p-1}_{\lambda}+u^{p-1}_\lambda-\vartheta(u_\lambda)+\chi(\epsilon)\nonumber\\
		&<&[\xi(z)+\hat{\xi}_\rho]u^{p-1}_{\lambda}+\lambda_0u^{p-1}_\lambda+f(z,u_\lambda)+\chi(\epsilon)\ (\mbox{see (\ref{eq40})})\nonumber\\
		&=&[\xi(z)+\hat{\xi}_\rho]u^{p-1}_\lambda+\lambda u^{q-1}_\lambda+f(z,u_\lambda)-(\lambda-\lambda_0)u^{q-1}_\lambda+\chi(\epsilon)\nonumber\\
		&<&[\xi(z)+\hat{\xi}_\rho]u^{p-1}_\lambda+\lambda u_\lambda^{q-1}+f(z,u_\lambda)\ \mbox{for}\ \epsilon>0\ \mbox{small enough}\nonumber\\
		&=&-{\rm div}\,a(Du_\lambda)+[\xi(z)+\hat{\xi}_p]u^{p-1}_\lambda-\vartheta(u_\lambda)\ \mbox{for almost all}\ z\in\Omega_0\ (\mbox{recall that}\ u_\lambda\in S_\lambda).
	\end{eqnarray}
	
	Then from (\ref{eq40}) and Proposition \ref{prop4}, we see that for small enough $\epsilon>0$ we have
	$$u_\lambda-m^\epsilon_0\in{\rm int}\,\hat{C}_+(\overline{\Omega}_0),$$
	which contradicts the definition of $m_0$. Hence $\lambda\notin\mathcal{L}$ and so $\lambda^*\leq\lambda_0<+\infty$.
\end{proof}

By Propositions \ref{prop9} and \ref{prop10} it follows that
\begin{equation}\label{eq42}
	(0,\lambda^*)\subseteq\mathcal{L}\subseteq\left(0,\lambda^*\right].
\end{equation}
\begin{prop}\label{prop11}
	If hypotheses $H(a),H(\xi),H(\vartheta),H(f)$ hold and $\lambda\in(0,\lambda^*)$, then problem \eqref{eqp} admits at least two positive solutions
	$$u_0,\hat{u}\in{\rm int}\,C_+,\ u_0\neq\hat{u}.$$
\end{prop}

\begin{proof}
	Let $0<\mu<\lambda<\eta<\lambda^*$. We have $\mu,\eta\in\mathcal{L}$ (see (\ref{eq42})). On account of Proposition \ref{prop9} we can find $u_\mu\in S_\mu\subseteq{\rm int}\,C_+,\ u_0\in S_\lambda\subseteq{\rm int}\,C_+,\ u_\eta\in S_\eta\subseteq{\rm int}\,C_+$ such that
	\begin{eqnarray}\label{eq43}
		&&u_0-u_\eta\in{\rm int}\,C_+\ \mbox{and}\ u_\eta-u_0\in{\rm int}\,C_+,\nonumber\\
		&\Rightarrow&u_0\in {\rm int}_{C^1_0(\overline{\Omega})}[u_\mu,u_\eta].
	\end{eqnarray}
	
	With $\beta>||\xi||_\infty$, we introduce the Carath\'eodory function $d_\lambda(z,x)$ defined by
	\begin{equation}\label{eq44}
		d_\lambda(z,x)=\left\{\begin{array}{ll}
			\vartheta(u_\mu(z))+\lambda u_\mu(z)^{q-1}+f(z,u_\mu(z))+\beta u_\mu(z)^{p-1}&\mbox{if}\ x\leq u_\mu(z)\\
			\vartheta(x)+\lambda x^{q-1}+f(z,x)+\beta x^{p-1}&\mbox{if}\ u_\mu(z)<x.
		\end{array}\right.
	\end{equation}
	
	We set $D_\lambda(z,x)=\int^x_0 d_\lambda (z,s)ds$ and consider the functional $\varphi_\lambda:W^{1,p}_0(\Omega)\rightarrow\RR$ defined by
	$$\varphi_\lambda(u)=\int_\Omega G(Du)dz+\frac{1}{p}\int_\Omega[\xi(z)+\beta]|u|^pdz-\int_\Omega D_\lambda(z,u)dz\ \mbox{for all}\ u\in W^{1,p}_0(\Omega).$$
	
	We know that $\varphi_\lambda\in C^1(W^{1,p}_0(\Omega))$ (see Papageorgiou \& Smyrlis \cite[Proposition 3]{17}). Also, let
	\begin{equation}\label{eq45}
		\hat{d}_\lambda(z,x)=\left\{\begin{array}{ll}
			d_\lambda(z,x)&\mbox{if}\ x\leq u_\eta(z)\\
			d_\lambda(z,u_\eta(z))&\mbox{if}\ u_\eta(z)<x.
		\end{array}\right.
	\end{equation}
	
	This is a Carath\'eodory function. We set $\hat{D}_\lambda(z,x)=\int^x_0\hat{d}_\lambda(z,s)ds$ and consider the $C^1$-functional $\hat{\varphi}_\lambda:W^{1,p}_0(\Omega)\rightarrow\RR$ defined by
	$$\hat{\varphi}_\lambda(u)=\int_\Omega G(Du)dz+\frac{1}{p}\int_\Omega[\xi(z)+\beta]|u|^pdz-\int_\Omega\hat{D}_\lambda(z,u)dz\ \mbox{for all}\ u\in W^{1,p}_0(\Omega).$$
	
	Using (\ref{eq44}) and (\ref{eq45}) and the nonlinear regularity theory (see the proof of Proposition \ref{prop7}), we show that
	\begin{eqnarray}
		&&K_{\varphi_\lambda}\subseteq\left[u_\mu\right)\cap{\rm int}\,C_+,\label{eq46}\\
		&&K_{\hat{\varphi}_\lambda}\subseteq[u_\mu,u_\eta]\cap{\rm int}\,C_+.\label{eq47}
	\end{eqnarray}
	
	From (\ref{eq47}) we see that we can assume that
	\begin{equation}\label{eq48}
		K_{\hat{\varphi}_\lambda}=\{u_0\}
	\end{equation}
	or otherwise we already have a second positive solution for \eqref{eqp} (see (\ref{eq45})) and so we are done.
	
	Clearly, $\hat{\varphi}_\lambda(\cdot)$ is coercive (see (\ref{eq45})) and sequentially weakly lower semicontinuous. So, we can find $\hat{u}_0\in W^{1,p}_0(\Omega)$ such that
	\begin{eqnarray}\label{eq49}
		&&\hat{\varphi}_\lambda(\hat{u_0})=\inf\{\hat{\varphi}_\lambda(u):u\in W^{1,p}_0(\Omega)\},\\
		&\Rightarrow&\hat{u}_0\in K_{\hat{\varphi}_\lambda},\nonumber\\
		&\Rightarrow&\hat{u}_0=u_0\ (\mbox{see (\ref{eq48})}).\nonumber
	\end{eqnarray}
	
	But from (\ref{eq44}) and (\ref{eq45}) we see that
	\begin{equation}\label{eq50}
		\left.\hat{\varphi}_\lambda\right|_{[u_\mu,u_\eta]}=\left.\varphi_\lambda\right|_{[u_\mu,u_\eta]}.
	\end{equation}
	
	It follows from (\ref{eq43}), (\ref{eq49}), (\ref{eq50}) that
	\begin{eqnarray}\label{eq51}
		&&u_0\ \mbox{is a local}\ C^1_0(\overline{\Omega})\mbox{-minimizer of}\ \varphi_\lambda,\nonumber\\
		&\Rightarrow&u_0\ \mbox{is a local}\ W^{1,p}_0(\Omega)\mbox{-minimizer of}\ \varphi_\lambda\ (\mbox{see \cite{5}}).
	\end{eqnarray}
	
	On account of (\ref{eq44}) and (\ref{eq46}), we may assume that
	\begin{equation}\label{eq52}
		K_{\varphi_\lambda}\ \mbox{is finite}.
	\end{equation}
	
	Otherwise we already have an infinity of positive smooth solutions. From (\ref{eq51}), (\ref{eq52}) and Theorem 5.7.6 of Papageorgiou, R\u{a}dulescu \& Repov\v{s} \cite{15}, we see that we can find small $\rho\in(0,1)$  such that
	\begin{equation}\label{eq53}
		\varphi_\lambda(u_0)<\inf\{\varphi_\lambda(u):||u-u_0||=\rho\}=m_\rho.
	\end{equation}
	
	Hypothesis $H(f)(ii)$ and Corollary \ref{cor3} imply that if $u\in{\rm int}\,C_+$, then
	\begin{equation}\label{eq54}
		\varphi_\lambda(tu)\rightarrow-\infty\ \mbox{as}\ t\rightarrow+\infty.
	\end{equation}
	\begin{claim}
		$\varphi_\lambda$ satisfies the $C$-condition.
	\end{claim}
	
	Consider a sequence $\{u_n\}_{n\geq 1}\subseteq W^{1,p}_0(\Omega)$ such that
	\begin{eqnarray}
		&&|\varphi_\lambda(u_n)|\leq c_{15}\ \mbox{for some}\ c_{15}>0,\ \mbox{and all}\ n\in\NN,\label{eq55}\\
		&&(1+||u_n||)\varphi'_\lambda(u_n)\rightarrow\ \mbox{in}\ W^{-1,p'}(\Omega)=W^{1,p}_0(\Omega)^*\ \mbox{as}\ n\rightarrow\infty.\label{eq56}
	\end{eqnarray}
	
	From (\ref{eq56}) we have
	\begin{eqnarray}\label{eq57}
		&&\left|\left\langle A(u_n),h\right\rangle+\int_\Omega[\xi(z)+\beta]|u_n|^{p-2}u_nhdz-\int_\Omega d_\lambda(z,u_n)hdz\right|\leq\frac{\epsilon_n||h||}{1+||u_n||}\\
		&&\mbox{for all}\ h\in W^{1,p}_0(\Omega),\ \mbox{with}\ \epsilon_n\rightarrow 0^+.\nonumber
	\end{eqnarray}
	
	In (\ref{eq57}) we choose $h=-u^-_n\in W^{1,p}_0(\Omega)$. From (\ref{eq44}) and Lemma \ref{lem2}, we have
	\begin{eqnarray}\label{eq58}
		&&\frac{c_1}{p-1}||Du^-_n||^p_p+\int_\Omega[\xi(z)+\beta](u^-_n)^pdz\leq\epsilon_n+c_{16}||u^-_n||\nonumber\\
		&&\mbox{for some}\ c_{16}>0,\ \mbox{and all}\ n\in\NN,\nonumber\\
		&\Rightarrow&\{u^-_n\}_{n\geq 1}\subseteq W^{1,p}_0(\Omega)\ \mbox{is bounded}\ (\mbox{recall that}\ \beta>||\xi||_\infty).
	\end{eqnarray}
	
	Next, in (\ref{eq57}) we choose $h=u^+_n\in W^{1,p}_0(\Omega)$. Then
	\begin{eqnarray}\label{eq59}
		&&-\int_\Omega(a(Du^+_n),Du^+_n)_{\RR^N}dz-\int_\Omega[\xi(z)+\beta](u^+_n)^pdz+\int_\Omega[\lambda(u^+_n)^q+f(z,u^+_n)u^+_n]dz\leq c_{17}\\
		&&\mbox{for some}\ c_{17}>0\ \mbox{and all}\ n\in\NN\ (\mbox{see (\ref{eq44}) and hypothesis}\ H(\vartheta)(ii)).\nonumber
	\end{eqnarray}
	
	From (\ref{eq55}) and (\ref{eq58}) we obtain
	\begin{eqnarray}\label{eq60}
		&&\int_\Omega pG(Du^+_n)dz+\int_\Omega[\xi(z)+\beta](u^+_n)^pdz-\int_\Omega\left[\frac{\lambda p}{q}(u^+_n)^q+pF(z,u^+_n)\right]dz\leq c_{18}\\
		&&\mbox{for some}\ c_{18}>0\ \mbox{and all}\ n\in\NN.\nonumber
	\end{eqnarray}
	
	Adding (\ref{eq59}) and (\ref{eq60}) and using hypothesis $H(a)(iv)$, we obtain
	\begin{eqnarray}\label{eq61}
		&&\int_\Omega[f(z,u^+_n)u^+_n-pF(z,u^+_n)]dz\leq c_{19}+\lambda\left[\frac{p}{q}-1\right]||u^+_n||^q_q\\
		&&\mbox{for some}\ c_{19}>0,\ \mbox{all}\ n\in\NN.\nonumber
	\end{eqnarray}
	
	From hypotheses $H(f)(i),(iii)$ we see that we can find $\hat{\beta}_1\in(0,\hat{\beta_0})$ and $c_{20}>0$ such that
	\begin{equation}\label{eq62}
		\hat{\beta}_1x^{\sigma}-c_{20}\leq f(z,x)x-pF(z,x)\ \mbox{for almost all}\ z\in\Omega\ \mbox{and all}\ x\geq 0.
	\end{equation}
	
	Using (\ref{eq62}) in (\ref{eq61}) and recalling that $q<\sigma$ (see hypothesis $H(f)(iii)$) we obtain that
	\begin{equation}\label{eq63}
		\{u^+_n\}_{n\geq 1}\subseteq L^{\sigma}(\Omega)\ \mbox{is bounded}.
	\end{equation}
	
	First, suppose that $N\neq p$. It is clear from hypothesis $H(f)(iii)$ that we may assume that $\sigma<r<p^*$ (recall that $p^*=+\infty$ if $N\leq p$). Let $t\in(0,1)$ be such that
	$$\frac{1}{r}=\frac{1-t}{\sigma}+\frac{t}{p^*}.$$
	
	From the interpolation inequality (see, for example, Papageorgiou \& Winkert \cite[Proposition 2.3.17, p.116]{19}), we have
	\begin{eqnarray}\label{eq64}
		&&||u^+_n||_r\leq||u^+_n||^{1-t}_\sigma||u^+_n||^t_{p^*},\nonumber\\
		&\Rightarrow&||u^+_n||^r_r\leq c_{21}||u^+_n||^{tr}\\
		&&\mbox{for some}\ c_{21}>0\ \mbox{and all}\ n\in\NN\ (\mbox{see (\ref{eq63})}).\nonumber
	\end{eqnarray}
	
	From hypothesis $H(f)(i)$, we have
	\begin{equation}\label{eq65}
		f(z,x)x\leq c_{22}[1+x^r]\ \mbox{for almost all}\ z\in\Omega,\ \mbox{all}\ x\geq 0\ \mbox{and some}\ c_{22}>0.
	\end{equation}
	
	In (\ref{eq57}) we choose $h=u^+_n\in W^{1,p}_0(\Omega)$ and use Lemma \ref{lem2}. Then
	\begin{eqnarray}\label{eq66}
		&&\frac{c_1}{p-1}||Du^+_n||^p_p+\int_\Omega[\xi(z)+\beta](u^+_n)^pdz\leq\epsilon_n+\int_\Omega d_\lambda(z,u_n)u^+_ndz,\nonumber\\
		&\Rightarrow&\frac{c_1}{p-1}||Du^+_n||^p_p\leq c_{23}+\int_\Omega[\lambda(u^+_n)^q+f(z,u^+_n)u^+_n]dz\nonumber\\
		&&\mbox{for some}\ c_{23}>0\ \mbox{and all}\ n\in\NN\ (\mbox{see (\ref{eq44})})\nonumber\\
		&\leq&c_{24}[1+\lambda||u^+_n||^q+||u^+_n||^{tr}]\\
		&&\mbox{for some}\ c_{24}>0\ \mbox{and all}\ n\in\NN\ (\mbox{see (\ref{eq64}) and (\ref{eq65})}).\nonumber
	\end{eqnarray}
	
	The hypothesis on $\sigma$ (see $H(f)(iii)$) implies that $tr<p$. Also we have $q<p$. Therefore it follows from (\ref{eq66}) that
	\begin{equation}\label{eq67}
		\{u^+_n\}_{n\geq 1}\subseteq W^{1,p}_0(\Omega)\ \mbox{is bounded}.
	\end{equation}
	
	If $p=N$, then $p^*=+\infty$ and by the Sobolev embedding theorem, we have that $W^{1,p}_0(\Omega)\hookrightarrow L^s(\Omega)$ for all $1\leq s<\infty$. So, we need to replace in the previous argument  $p^*$ by $s>r>\sigma$ big enough. More precisely, as before, let $t\in(0,1)$ be such that
	\begin{eqnarray*}
		&&\frac{1}{r}=\frac{1-t}{\sigma}+\frac{t}{s},\\
		&\Rightarrow&tr=\frac{s(r-\sigma)}{s-\sigma}\rightarrow r-\sigma\ \mbox{as}\ s\rightarrow+\infty.
	\end{eqnarray*}
	
	Recall that $r-\sigma<p$ (see hypothesis $H(f)(iii)$). Hence for large enough $s>r$
	$$tr=\frac{s(r-\sigma)}{s-\sigma}<p.$$
	
	Then for such large $s>r$, the previous argument is valid and we again obtain (\ref{eq67}).
	
	From (\ref{eq58}) and (\ref{eq67}) we have that $\{u_n\}_{n\geq 1}\subseteq W^{1,p}_0(\Omega)$ is bounded. So, we may assume that
	\begin{equation}\label{eq68}
		u_n\stackrel{w}{\rightarrow}u\ \mbox{in}\ W^{1,p}_0(\Omega).
	\end{equation}
	
	In (\ref{eq57}) we choose $h=u_n-u\in W^{1,p}_0(\Omega)$, pass to the limit as $n\rightarrow\infty$, and use (\ref{eq68}). Then
	\begin{eqnarray*}
		&&\lim\limits_{n\rightarrow\infty}\left\langle A(u_n),u_n-u\right\rangle=0,\\
		&\Rightarrow&u_n\rightarrow u\ \mbox{in}\ W^{1,p}_0(\Omega)\\
		&&(\mbox{using the}\ (S)_+\ \mbox{property of}\ A(\cdot),\ \mbox{see Section 2}),\\
		&\Rightarrow&\varphi_\lambda(\cdot)\ \mbox{satisfies the}\ C-\mbox{condition}.
	\end{eqnarray*}
	
	This proves Claim 1.
	
	From (\ref{eq53}), (\ref{eq54}) and  Claim 1, we see that we can apply the mountain pass theorem. So, we can find $\hat{u}\in W^{1,p}_0(\Omega)$ such that
	\begin{equation}\label{eq69}
		\hat{u}\in K_{\varphi_\lambda}\subseteq\left[u_\mu\right)\cap{\rm int}\,C_+\ \mbox{(see (\ref{eq46})) and}\ m_\rho\leq\varphi_\lambda(\hat{u})\ (\mbox{see (\ref{eq53})}).
	\end{equation}
	
It follows 	rom (\ref{eq44}) and (\ref{eq69}) that
	$$\hat{u}\in S_\lambda\subseteq{\rm int}\,C_+\ \mbox{and}\ u_0\neq\hat{u}.$$
The proof is now complete.
\end{proof}
\begin{prop}\label{prop12}
	If hypotheses $H(a),H(\xi),H(\vartheta),H(f)$ hold, then $\lambda^*\in\mathcal{L}$.
\end{prop}
\begin{proof}
	Let $\{\lambda_n\}_{n\geq 1}\subseteq(0,\lambda^*)$ be such that $\lambda_n\uparrow\lambda^*$. We know that $\lambda_n\in\mathcal{L}$ for all $n\in\NN$ and so we can find $u_n=u_{\lambda_n}\in S_{\lambda_n}\subseteq{\rm int}\, C_+$ $(n\in\NN)$ increasing (see Proposition \ref{prop9}).
	
	Let $\hat{\varphi}_{\lambda_n}(\cdot)$ be the functional from the proof of Proposition \ref{prop11}, with $u_\mu=u_{n-1},\ u_\mu=u_{n+1}(n\geq 2)$. Then we have
	\begin{eqnarray}\label{eq70}
		\hat{\varphi}_{\lambda_n}(u_n)&\leq&\hat{\varphi}_{\lambda_n}(u_{n-1})\nonumber\\
		&=&\int_\Omega G(Du_{n-1})dz+\frac{1}{p}\int_\Omega[\xi(z)+\beta]u^p_{n-1}dz-\int_\Omega[\vartheta(u_{n-1})+\lambda_nu^{q-1}_{n-1}+\nonumber\\
		&&f(z,u_{n-1}+\beta u^{p-1}_{n-1})]u_{n-1}dz\nonumber\\
		&\leq&\int_{\Omega}G(Du_{n-1})dz+\frac{1}{p}\int_\Omega[\xi(z)+\beta]u^p_{n-1}dz-\int_\Omega[\vartheta(u_{n-1})+\lambda_{n-1}u^{q-1}_{n-1}+\nonumber\\
		&&f(z,u_{n-1})+\beta u^{p-1}_{n-1}]u_{n-1}dz\nonumber\\
		&\leq&\int_\Omega(a(Du_{n-1}),Du_{n+1})dz+\int_\Omega\xi(z)u^p_{n-1}dz-\int_{\Omega}[\vartheta(u_{n-1})+\lambda_{n-1}u^{q-1}_{n-1}+\nonumber\\
		&&f(z,u_{n-1})]u_{n-1}dz\ (\mbox{see (\ref{eq3}) and recall that}\ \beta>||\xi||_\infty)\nonumber\\
		&=&0\ (\mbox{since}\ u_{n-1}\in S_{\lambda_{n-1}}).
	\end{eqnarray}
	
	Also, we have
	\begin{eqnarray}\label{eq71}
		&&\left\langle A(u_n),h\right\rangle+\int_\Omega[\xi(z)+\beta]u^{p-1}_nhdz=\int_\Omega d_{\lambda_n}(z,u_n)hdz\\
		&&\mbox{for all}\ h\in W^{1,p}_0(\Omega)\ \mbox{and all}\ n\in\NN.\nonumber
 	\end{eqnarray}
	
	Using (\ref{eq70}), (\ref{eq71}) and reasoning as in the proof of Proposition \ref{prop11} (see Claim 1), we obtain that
		$$\{u_n\}_{n\geq 1}\subseteq W^{1,p}_0(\Omega)\ \mbox{is bounded}.$$
		
		From this, as in the proof of Proposition \ref{prop11}, exploiting the $(S)_+$ property of $A(\cdot)$, we obtain
		\begin{equation}\label{eq72}
			u_n\rightarrow u_*\ \mbox{in}\ W^{1,p}_0(\Omega).
		\end{equation}
		
		Passing to the limit as $n\rightarrow\infty$ in (\ref{eq71}) and using (\ref{eq72}), we have
		$$u_*\in S_{\lambda_*}\subseteq{\rm int}\,C_+\ \mbox{and so}\ \lambda^*\in\mathcal{L}.$$
The proof is now complete.
\end{proof}

This proposition implies that
$$\mathcal{L}=\left(0,\lambda^*\right].$$

Summarizing the situation for problem \eqref{eqp}, we can state the following bifurcation-type result.
\begin{theorem}\label{th13}
	If hypotheses $H(a),H(\xi),H(\vartheta),H(f)$ hold, then there exists $\lambda^*>0$ such that
	\begin{itemize}
		\item[(a)] for all $\lambda\in(0,\lambda^*)$ problem \ref{eqp} has at least two positive solutions
		$$u_0,\hat{u}\in{\rm int}\,C_+,\ u_0\neq\hat{u};$$
		\item[(b)] for $\lambda=\lambda^*$ problem \eqref{eqp} has at least one positive solution $u_*\in{\rm int}\, C_+$;
		\item[(c)] for all $\lambda>\lambda^*$ problem \eqref{eqp} has no positive solutions.
	\end{itemize}
\end{theorem}

\medskip
{\bf Acknowledgments.} This research was supported by the Slovenian Research Agency grants
P1-0292, J1-8131, J1-7025, N1-0064, and N1-0083.


\begin{thebibliography}{99}

\bibitem{1} L. Cherfils, Y. Ilyasov,  On the stationary solution of generalized reaction diffusion equations with $p\, \& \, q$ Laplacian, {\it Comm. Pure Appl. Anal.} {\bf 4} (2005), 9-22.

\bibitem{2} Y. Chu, R. Gao, Y. Sun, Existence and regularity of solutions to a quasilinear elliptic problem involving variable sources, {\it Bound. Value Probl.} {\bf 2017}:155 (2017).

\bibitem{3} L. Damascelli, Comparison theorems for some quasilinear degenerate elliptic operators and applications to symmetry and monotonicity results, {\it Ann. Inst. H. Poincar\'e Analyse Non Lin\'eaire} {\bf 15} (1998), 493-516.

\bibitem{4} L. Gasinski, N.S. Papageorgiou, {\it Nonlinear Analysis}, Chapman \& Hall/CRC, Boca Raton, FL, 2006.

\bibitem{5} J. Giacomoni, I. Schindler, P. Takac, Sobolev versus H\"older local minimizers and existence of multiple solutions for a singular quasilinear equation, {\it Ann. Sc. Normale Super. Pisa}, Ser. V {\bf 6} (2007), 117-156.

\bibitem{6} D. Gilbarg, N.S. Trundinger, {\it Elliptic Partial Differential Equations of Second Order}, Springer, Berlin, 1998.

\bibitem{7} S. Hu, N.S. Papageorgiou, {\it Handbook of Multivalued Analysis. Volume I: Theory}, Kluwer Academic Publishers, Dordrecht, 1997.

\bibitem{8} O. Ladyzhenskaya, N. Uraltseva, {\it Linear and Quasilinear Elliptic Equations}, Academic Press, New York, 1968.

\bibitem{9} A. Lazer, P.J. McKenna,  On a singular nonlinear elliptic boundary value problem, {\it Proc. Amer. Math. Soc.} {\bf 111} (1991), 721-730.

\bibitem{11} Q. Li, W. Gao, Existence of weak solutions to a class of singular elliptic equations, {\it Mediterr. J. Math.} {\bf 13} (2016), 4917-4927.
    
\bibitem{10} G. Lieberman, The natural generalization of the natural conditions of Ladyzhenskaya and Uraltseva for elliptic equations, {\it Comm. Partial Diff. Equations} {\bf 16} (1991), 311-361.

\bibitem{12} A. Mohammed,  Positive solutions of the $p$-Laplacian equation with singular nonlinearity, {\it J. Math. Anal. Appl.} {\bf 352} (2009), 234-245.

\bibitem{13} N.S. Papageorgiou, V.D. R\u adulescu, Coercive and noncoercive nonlinear Neumann problems with indefinite potential, {\it Forum Math.} {\bf 28} (2016), 545-571.

\bibitem{14} N.S. Papageorgiou, V.D. R\u adulescu, D.D. Repov\v s, Positive solutions for nonlinear parametric singular Dirichlet problems, {\it Bull. Math. Sci.} {\bf 9}:3 (2019), art. 1950011, 21 pp. 

\bibitem{15} N.S. Papageorgiou, V.D. R\u adulescu, D.D. Repov\v s, {\it Nonlinear Analysis -- Theory and Methods}, Springer Monographs in Mathematics, Springer, Cham, 2019.

\bibitem{16} N.S. Papageorgiou, V.D. R\u adulescu, D.D. Repov\v s,  Nonlinear nonhomogeneous singular problems, 
{\it Calc. Var. Partial Differential Equations} {\bf 59}:1 (2020), art. 9, 31 pp. 

\bibitem{17} N.S. Papageorgiou, G. Smyrlis, A bifurcation-type theorem for singular nonlinear elliptic equations, {\it Methods Appl. Anal.} {\bf 22} (2015), 147-170.

\bibitem{19} N.S. Papageorgiou, P. Winkert, {\it Applied Nonlinear Functional Analysis},  De Gruyter, Berlin, 2018.

\bibitem{18} N.S. Papageorgiou, P. Winkert, Singular $p$-Laplacian equations with superlinear perturbation, {\it J. Differential Equations} {\bf 266} (2019), 1462-1487.

\bibitem{20} K. Perera, Z. Zhang, Multiple positive solutions of singular $p$-Laplacian problems by variational methods, {\it Bound. Value Probl.} {\bf 2005}:3 (2005).

\bibitem{21} P. Pucci, J. Serrin, {\it The Maximum Principle}, Birkh\"auser, Basel, 2007.

\bibitem{22} V.V. Zhikov, Averaging of functionals of the calculus of variations and elasticity theory, {\it Math. USSR-Izvest.} {\bf 29} (1987), 33-66.

\end{thebibliography}
\end{document}